\documentclass[letterpaper,twosided,11pt]{article}
\title{Vanishing Cohomology of Dominant Line Bundles for Real Groups} 
\author{Jack A. Cook}
\date{}
\usepackage{authblk}
\usepackage{amssymb}
\usepackage{amsmath}
\usepackage{stmaryrd}
\usepackage{mathrsfs}
\usepackage{hyperref} 
\usepackage{titling}
\usepackage{bbm}
\usepackage{dsfont} 
\usepackage{ bbold }
\usepackage[T1]{fontenc}
\usepackage{palatino}
\usepackage{mathpazo}
\usepackage{marvosym} 
\usepackage{tikz-cd}
\usepackage[margin=1in]{geometry} 
\usepackage{ragged2e}
\usepackage[utf8]{inputenc}
\usepackage[english]{babel} 
\usepackage{blindtext}

\usepackage{amsthm}
\usepackage{pgfplots}
\usepackage{tikz-3dplot}
\usepackage{pgffor}
\usepackage{comment}
\usepackage{setspace}
\usepackage{asymptote}
\usepackage{color}   
\usepackage{hyperref}
\hypersetup{
    colorlinks,
    citecolor=blue,
    filecolor=black,
    linkcolor=black,
    urlcolor=black
}
\usepackage[bottom]{footmisc}
\usepackage{graphicx}

\usetikzlibrary{shapes,calc,positioning}
\usetikzlibrary{calc,intersections}
\usetikzlibrary{bending}
	
\theoremstyle{plain} 

\newtheorem*{ntheorem}{Theorem}
\newtheorem*{nlemma}{Lemma}

\newtheorem{theorem}{Theorem}[section]

\newtheorem{lemma}[theorem]{Lemma}
\newtheorem{proposition}[theorem]{Proposition}
\newtheorem{corollary}[theorem]{Corollary}

\theoremstyle{definition}
\newtheorem{definition}{Definition}
 
\newtheorem{example}{Example}

\theoremstyle{remark} 
\newtheorem{remark}{Remark}

\newcommand{\C}{\mathbb{C}}

\newcommand{\HH}{\mathbb{H}}

\newcommand{\R}{\mathbb{R}}
\newcommand{\Z}{\mathbb{Z}}

\newcommand{\KSpRes}{\widetilde{\mathcal{N}_\theta}}

\newcommand{\close}[1]{\overline{#1}}
\newcommand{\ds}{\oplus}
\newcommand{\tensor}{\otimes}
\newcommand{\Sym}{\operatorname{Sym}}
\newcommand{\im}{\operatorname{Im }}

\newcommand{\Hom}{\operatorname{Hom}}
\newcommand{\Rep}{\operatorname{Rep}}

\newcommand{\dWts}[1]{\mathbb{X}_+^*(#1)}

\newcommand{\lie}[1]{\mathfrak{#1}}

\newcommand{\ad}{\operatorname{ad}}
\newcommand{\Ad}{\operatorname{Ad}}

\usetikzlibrary{hobby}

\begin{document}
\maketitle

\begin{abstract} 
	In \cite{Broer1993}, it was shown that certain line bundles on $\widetilde{\mathcal{N}}=T^*G/B$ have vanishing higher cohomology. We prove a generalization of this theorem for real reductive algebraic groups. More specifically, if $\mathcal{N}_\theta$ denotes the cone of nilpotent elements in a Cartan subspace $\lie{p},$ we have a similar construction of a resolution of singularities $\KSpRes.$  We prove that for a certain cone of weights $H^i(\KSpRes,\mathcal{O}_{\KSpRes}(\lambda))=0$ for $i> 0.$ This follows by combining a simple calculation of the canonical bundle for $\KSpRes$ with Grauert--Riemenschneider vanishing. Restricting to the structure sheaf, we get a characterization of the singularities of the normalization of $\mathcal{N}_\theta.$ We use this to show that for groups of QCT (Definition \ref{Quasi_complex}), $\C[\mathcal{N}_\theta]$ is equivalent as a $K$-representation to a certain cohomologically induced module giving a new proof of a result in $\cite{KostantRallis1971}.$   
\end{abstract}

\tableofcontents

\section{Introduction}
	Let $G$ be a connected complex reductive algebraic group. Fix a real form $G_\R.$ Let $\theta$ be the Cartan involution of $G$ (with respect to $G_\R$) and $K_\theta=G^\theta$ the set of fixed points and the associated Cartan decomposition of the Lie algebra $\lie{g}=\lie{k}\ds \lie{p}.$ By observing the bracket relation $[\lie{k},\lie{p}]\subseteq \lie{p}$, it is clear that $K_\theta$ acts on $\lie{p}$ via the adjoint representation. By \cite[Theorem 2]{KostantRallis1971}, there are finitely many $K_\theta$-orbits of this action on the subvariety $\mathcal{N}_\theta\subseteq \lie{p}$ of $\ad$-nilpotent elements. Note that $K_\theta$ may be disconnected even though $G$ is connected. For the purposes of the proofs presented here we will deal with $K:=(K_\theta)_0$ the identity component of $K_\theta.$ 
	
	Under these assumptions, $K$ will have finitely many orbits of maximal dimension on $\mathcal{N}_\theta.$ These are the irreducible components of $\mathcal{N}_\theta$. Pick a single $K$-orbit (irreducible component) which we shall call a principal $K$-orbit and denote it $\mathcal{O}_{prin}^K$ or $\mathcal{O}_{prin}$ if the group is not ambiguous. We call any element $X\in \mathcal{O}_{prin}^K$ a \textit{principal nilpotent element} or principal for short. Now, fix any $X\in \mathcal{O}_{prin}^K.$ As $X\in \lie{p}\subset \lie{g}$ we may also consider the $G$-orbit $\mathcal{O}^G_X=G\cdot X.$ A principal $K$-orbit is a Lagrangian subvariety in the $G$-orbit $\mathcal{O}_X^G$ (\cite[Corollary 5.18]{Vogan1991}). Note that the $G$-nilpotent cone $\mathcal{N}$ has its own principal open orbit which we denote $\mathcal{O}_{prin}^G$. The orbit $\mathcal{O}_X^G$ need \textit{not} be the principal $G$-orbit (in fact it is generally not so). Further, \[ \mathcal{N}_\theta=\{ \xi\in \mathcal{N}: \theta\xi=-\xi \}=\mathcal{N}\cap \lie{p} \] Denote by $\mathcal{N}_\theta':=\close{K\cdot X}=\close{\mathcal{O}_{prin}^K}.$ This is our variety of interest.  
		
	By way of the Jacobson-Morozov theorem, we may complete $X$ to an $\lie{sl}(2,\C)$-triple $\{H,X,Y\}$ which can be chosen such that $H\in \lie{k}$ \cite[Proposition 4]{KostantRallis1971}. Further, as $X$ is principal nilpotent  it is shown in the discussion before Proposition 13 in loc. cit. that $\ad H$ has only even integer eigenvalues. Decomposing $\lie{g}$ into $\ad H$-eigenspaces $\lie{g}_i$, we obtain that $\lie{g}=\sum_{i\in \Z} \lie{g}_{2i}.$ We define three subalgebras \begin{align*} \lie{q}=\bigoplus_{i\geq 0} \lie{g}_i && \lie{l}=\lie{g}_0 && \lie{u}=\bigoplus_{i>0} \lie{g}_i.  \end{align*}  $\lie{q}$ is a parabolic subalgebra of $\lie{g}$ with Levi decomposition $\lie{q}=\lie{l}\ds \lie{u}.$ Note that $X\in \lie{g}_2\subseteq \lie{u}.$
	 	
	It is now a somewhat non-trivial result that $Q=LU$ (the connected subgroups of $G$ corresponding to $\lie{q}= \lie{l}\ds \lie{u}$) has a single open dense orbit in $\lie{u}$ \cite[Lemma 4.1.4]{Collingwood1993}. Define $\widetilde{\mathcal{O}^G_X}=G\times_Q \lie{u}.$ Using this open dense $Q$ orbit on $\lie{u}$, we obtain a resolution of singularities of the closure $\close{\mathcal{O}_X^G}$ \[\begin{tikzcd} \widetilde{\mathcal{O}^G_X}\arrow[d,swap,"\mu"]\arrow[r,"\pi"]  & G/Q \\ \close{\mathcal{O}_X^G}    \end{tikzcd}  \]
The vertical arrow is given by the adjoint action map, is birational, and is an isomorphism over $\mathcal{O}$.   

Now let $\xi\in \Rep(Q)$ be any finite dimensional representation of $Q.$ Then define \[W_\xi=G\times_Q \xi^*\] to be the homogeneous vector bundle over $G/Q.$ If $\xi$ is one-dimensional, then we denote the sheaf of sections of this line bundle as $\mathcal{O}_{G/Q}(\xi).$ Denote the set of characters of $L$ as $\mathbb{X}^*(L).$ As pulling back sheaves preserves rank, we have a family of rank one $G$-equivariant sheaves on $\widetilde{\mathcal{O}_X^G}$ which we will denote by \[ \{\mathcal{O}_{\widetilde{\mathcal{O}}}(\xi)\; :\;  \xi\in \mathbb{X}^*(L)\}\]
     
Notice further that if $X$ is $G$-principal nilpotent (see discussion after Lemme 5.2 in \cite{Kostant1959} for a definition), then $Q=B$ is a Borel subgroup and $\lie{u}=\lie{n}.$ Thus $\widetilde{\mathcal{O}}=\widetilde{\mathcal{N}}=T^*(G/B)$ is the cotangent bundle and the moment map here is the \textit{Springer Resolution}.  

Broer in \cite{Broer1993}, proves the following vanishing result for this resolution: 

\begin{ntheorem} 
	Let $B=HN$ be a Levi decomposition of the Borel subgroup defined by $X$. Let $\lambda\in \mathbb{X}^*(H)$ be dominant (with respect to $B$).  Extend this to a one dimensional representation of $B.$ Then \[H^i(\widetilde{\mathcal{N}}, \mathcal{O}_{\widetilde{\mathcal{N}}}(\lambda))=0\;\;\; i>0. \] 
	\end{ntheorem}

By the Grothendieck spectral sequence, this result is equivalent to saying that \[ H^0(\mathcal{N}, R^i\mu_*\mathcal{O}_{\widetilde{\mathcal{N}}}(\lambda))=0\;\;\; i>0. \] 

In a later paper \cite{Broer1994}, Broer reproves the above theorem with simpler methods. In doing so, he gives a simple proof of the following: 
\begin{nlemma} 
	Let $\omega_{\widetilde{\mathcal{N}}}$ denote the canonical bundle of the contangent bundle. Then \[ \omega_{\widetilde{\mathcal{N}}}\simeq \mathcal{O}_{\widetilde{\mathcal{N}}}.\]
	\; \\
\end{nlemma}

	We now return to the $K$-setting. As we pick $X\in \lie{p},$ we may choose $H\in \lie{k}$ and thus the resolution above restricts to an associated resolution \[\begin{tikzcd} \KSpRes'=K\times_{Q\cap K} (\lie{u}\cap \lie{p}) \arrow[d,swap, "\mu_K"]\arrow[r, "\pi_K"]  & K/Q\cap K \\ \mathcal{N}_\theta'    \end{tikzcd}  \]
	The vertical arrow here is again the adjoint action. We call this resolution the $K$-\textit{Springer Resolution}. 
	
	Let $B=HN$ be a $\theta$-stable Borel subgroup of $G$ contained in $Q.$ Then $B\cap K$ is a Borel subgroup in $K;$ in particular $T:=H^\theta$ is a Cartan subgroup of $K.$ Let $\mathbb{X}^*(T)$ denote the character lattice of $T.$ $B\cap K$ determines a set of dominant characters (weights) which we will denote $\mathbb{X}_+^*(T).$ For each $\lambda \in  \dWts{T},$  let $V_\lambda$ be the irreducible finite dimensional representation of $K$ of highest weight $\lambda.$  
	
	\begin{definition} 
		We say a weight $\lambda\in \mathbb{X}_+^*(T)$ is $Q\cap K$-\textbf{dominant} if there is a one-dimensional subspace stabilized by $Q\cap K$ in $V_\lambda.$  Denote the monoid of all $Q\cap K$-dominant weights as $\mathbb{W}(Q\cap K).$    
	\end{definition}

	We now have the main results. 

	\begin{theorem}[Corollary \ref{Nilpotent_Orbit_Case}] 
		If $Q=LU$ denotes the $\theta$-stable parabolic associated to a principal nilpotent element $X$, then \[\omega_{\KSpRes}\simeq \mathcal{O}_{\KSpRes}(2\rho(\lie{u}\cap \lie{p})-2\rho(\lie{u}\cap \lie{k}))   \]
		where $2\rho(-)$ means the sum of positive weights of $T$ on the given space. 
	\end{theorem}   
	
	This follows from a more general computation for all conormal bundles to $K$-orbits on partial flag varieties. We may use the preceding result and combine it with Grauert-Riemenschnieder vanishing to obtain the main theorem: 
	
	\begin{theorem}[Theorem \ref{Vanishing_of_Higher_Cohomology}] 
		For $\lambda'=\lambda-2\rho(\lie{u}\cap \lie{p})+2\rho(\lie{u}\cap \lie{k}) \in \mathbb{W}(Q\cap K)$  we have \[H^i(\KSpRes',\mathcal{O}_{\KSpRes'}(\lambda'))=0 \text{ for } i>0 \]
	\end{theorem} 
As a consequence of this, we have the following characterization of singularities in the normalization of $\mathcal{N}_\theta.$ 

\begin{theorem}[Corollary \ref{Rational_normal} ]
	The normalization $\mathcal{N}_\theta^\nu$ of the nilpotent cone has rational singularities. 
\end{theorem} 

	Note that we are simply picking a single principal $K$-orbit in $\mathcal{N}_\theta$ for adjoint groups. We similarly could have proven this result for any $K$-orbit utilizing the resolution of singularities constructed in Section 3. All of the results hold in this more generic context as well.   \\

\textbf{Acknowledgements}: This paper emerged from the author's Ph.D Thesis. I thank Peter Trapa for introducing me to the orbit method and Lucas Mason-Brown for inspiring this particular result and identifying errors in earlier versions. I am also grateful to David Vogan for explaining the branching rules needed in the later parts of this work.

\section{Background from Nilpotent Orbits} 
	 Let $G$ be a connected complex reductive algebraic group. Pick a coadjoint nilpotent $G$-orbit $\mathcal{O}_\lambda$ for some $\lambda\in \mathcal{N}^*.$ As $\lie{g}$ is reductive, there exists a $G$-invariant non-degenerate bilinear form $\psi$ on $\lie{g}$ giving a canonical isomorphism $\lie{g}\cong \lie{g}^*.$ Therefore, $\lambda\leftrightarrow X_\lambda$ is a bijection between $\mathcal{N}^*$ and $\mathcal{N}$ the cone of nilpotent elements in $\lie{g}.$ 
By the Jacobson-Morozov theorem, we can extend $X_\lambda$ to a $\lie{sl}(2)$-triple which we will denote $\{ H_\lambda, X_\lambda, Y_\lambda\}\subseteq \lie{g}.$ We will drop the $\lambda$ subscript if there is no room for confusion. As $H_\lambda$ is semisimple and the weights of $\lie{sl}(2)$ are integral, we have that under $\ad H_\lambda$ our Lie algebra decomposes \[   \lie{g}=\bigoplus_{i\in\Z} \lie{g}_i\]
where $\lie{g}_i=\{ A\in \lie{g}: [H_\lambda,A]=iA\}$ and $[\lie{g}_i,\lie{g}_j]=\lie{g}_{i+j}.$ Notice at once that $\dim \lie{g}_i=\dim \lie{g}_{-i}$ as finite dimensional $\lie{sl}(2)$ representations have symmetric multiplicities for positive and negative weights.  

Consider the subalgebra generated by only non-negative weights of $H_\lambda.$ Namely, consider $\lie{q}=\bigoplus_{i\geq 0} \lie{g}_i.$ 

\begin{lemma} 
	$\lie{q}$ is a parabolic subalgebra of $\lie{g}.$
\end{lemma} 
\begin{proof} 
	Every $H_\lambda $ sits inside a Cartan subalgebra $\lie{h}.$ By the choice of subalgebra, we have that $\alpha(H_\lambda)\geq 0 \in \Z$ for all $\alpha\in \Delta^+(\lie{g},\lie{h}).$ Whence, $\lie{q}$ contains a Borel subalgebra (the span of these positive roots and the Cartan subalgebra). 
\end{proof} 

We have a decomposition of $\lie{q}=\lie{g}_0\ds \bigoplus_{i>0}\lie{g}_i=:\lie{l}\ds \lie{u}$ where $\lie{l}$ is reductive and $\lie{u}$ is nilpotent. Denote by $\close{\lie{u}}=\bigoplus_{i<0} \lie{g}_i$ the subalgebra of negative weights.   We set $Q$ to be the connected subgroup of $G$ with Lie algebra $\lie{q}.$ 

\begin{lemma}[Kostant {\cite[Section 4]{Kostant1963}} ] \label{Dense_Q_orbit_g_2} 
	Set $\lie{g}^2=\bigoplus_{i\geq 2} \lie{g}_i\subseteq \lie{q}.$ Then $X_\lambda\in \lie{g}^2$ and $Q\cdot X_\lambda\subseteq \lie{g}^2.$ Further, $Q\cdot X$ is open and dense in $\lie{g}^2.$     
\end{lemma}  
\begin{proof} 
	See \cite[Lemma 4.1.4]{Collingwood1993} for a proof. 
\end{proof} 

Now, the cotangent bundle of $G/Q$ is a $G$-equivariant vector bundle. The structure theory of $G$-equivariant vector bundles on a homogeneous space is an easy exercise. This implies that we can write \[ T^*(G/Q)=G\times_Q (T^*_{eQ} G/Q)=G\times_Q (\lie{g}/\lie{q})^*=G\times_Q (\bar{\lie{u}})^*\]
Again using the invariant form $\psi,$ we obtain an isomorphism: \[ G\times_Q (\bar{\lie{u}})^*\cong  G\times_Q \lie{u} \]
Consider the $G$-equivaraint subbundle of $G\times_Q \lie{u} $ defined by \[ \mathcal{R}:=G\times_Q \lie{g}^2. \]
We have the following lemma: 

\begin{lemma} 
	The adjoint map $\mu:\mathcal{R}\to \mathcal{N}$ given by $\mu([g,X])=\Ad(g)X$ is a proper birational map onto $\overline{G\cdot X_\lambda}.$  
\end{lemma} 
\begin{proof} 
	Define a morphism $m:\mathcal{R}\to G/Q\times \mathcal{N}$ by $(g,X)\mapsto (gQ, \Ad(g)(X))$. We claim this map is injective. To see this, suppose $(hQ,\Ad(h)(Y))=(gQ,\Ad(g)(X))$ for some $(h,Y), (g,X)\in \mathcal{R}.$ Then by definition $hQ=gQ\iff \exists q\in Q$ such that $g=hq$ or equivalently that $h^{-1}g=q\in Q.$ Additionally, $\Ad(h)^{-1}\Ad(g)X=Y.$ Now, \[ (g,X)\sim (hq,X)\sim (h,\Ad(q)X)\sim(h,\Ad(h)^{-1}\Ad(g)X)=(h,Y) \]
	Hence, $m$ is injective. Further, its image is closed as on fibres we have an injective linear map of finite dimensional spaces. Thus $m$ is proper.  We then see that the adjoint map $\mu$ is simply the restriction of the projection map $\pi_2:G/Q\times \mathcal{N}\to \mathcal{N}.$ As $G/Q$ is projective (hence proper over $\C$), $\pi_2$ is projective. By \cite[Theorem 4.9]{Hartshorne1977}, projective morphisms of noetherian schemes are proper. Hence, $\mu$ is proper. 
	
	To show it is birational, notice that $G\cdot X_\lambda$ is open in the image (as it is the orbit of an affine algebraic group). Now, the fibre of $\mu$ over $X_\lambda$ is $G^{X_\lambda}/(G^{X_\lambda}\cap Q)=\{*\}$ as $G^X$ is contained in $Q.$ Whence, on $G\cdot X_\lambda$ we can define a rational inverse, namely $X \mapsto (g,X_\lambda)$ for any $g$ such that $\Ad(g)X_\lambda=X.$ This is well defined up to right multiplication by the stabilizer $G^{X_{\lambda}}$ which is contained in $Q.$   Finally, as $\mu$ is proper the image is closed and thus \[\im(\mu)\supseteq \overline{\mathcal{O}_\lambda}\] The reverse inclusion follows from the definition of $\mu.$ Therefore, the image is $\overline{G\cdot X_\lambda}.$ This completes the proof.   
\end{proof} 

\begin{corollary} 
	We have an isomorphism in $G$-equivariant $\mathbb{K}$-theory \[ \mathbb{K}^G(\mathcal{R})\cong \mathbb{K}(Q)\]
	the latter being the representation ring of $Q.$ 
\end{corollary}
\begin{proof} 
	By Thomason's theorem \cite[Proposition 6.2]{Thomason1987} we have that \[\mathbb{K}^G(\mathcal{R})\cong \mathbb{K}^Q(\lie{g}^2)\]
	Now by \cite[Corollary 4.2]{Thomason1987} we have that if $G$ acts on affine $n$-space linearly, $\mathbb{K}^G(X)\cong \mathbb{K}^G(X\times \mathbb{A}_\C^n)$ for all $n$ and all $X.$ Setting $X=\{*\},$ we obtain that \[ \mathbb{K}^Q(\lie{g}^2)\cong \mathbb{K}^Q(\{*\})=\mathbb{K}(Q).\]
	This completes the proof.  
\end{proof}

\section{Nilpotent Orbits in the Real Case} 

As above, let $G$ be a connected complex reductive algebraic group. Further, let $\sigma_\R:G\to G$ and $\sigma_c:G\to G$ be involutive automorphisms defining real and compact forms respectively. Denote the real form as $G_\R:=G^{\sigma_\R}=G(\R).$ Using these, we define a Cartan involution \[ \theta=\sigma_\R\sigma_c:G\to G\]
We abuse notation and denote by $\theta$ the corresponding involution on the Lie algebra $\lie{g}.$ By further abuse of notation, we will also write $\theta$ when discussing the restriction to $\lie{g}_\R:=\lie{g}(\R)$ the Lie algebra of $G_\R.$ Set $K_\theta:=G^\theta.$ Then $K_\theta=(K_\R)_\C$ for $K_\R=G_\R^\theta.$ Again denote by $K$ the identity component of $K_\theta.$    

For the orbit method for real groups, we need to consider $\mathcal{N}^*_\theta=\mathcal{N}^*\cap (\lie{g}/\lie{k})^*$ instead of $\mathcal{N}^*.$ By fixing a non-degenerate invariant symmetric bilinear form $\psi$ on $\lie{g}$ we identify \[ \mathcal{N}_\theta \simeq \mathcal{N}_\theta^*\]
and remark that \[ \mathcal{N}_\theta=\{ \xi\in \mathcal{N}: \theta\xi=-\xi \}=\mathcal{N}\cap \lie{p} \]

By the results in \cite[Introduction p. 755]{KostantRallis1971}, we have that for any $X_\lambda\in \mathcal{N}_\theta$ we can complete this to an $\lie{sl}(2,\C)$ triple with $Y_\lambda\in \mathcal{N}_\theta$ and $H_\lambda\in \lie{k}.$ We will denote the $K$-orbit in $\mathcal{N}_\theta$ coming from the $G$-orbits $\mathcal{O}_\lambda$ as $\mathcal{O}^\lie{k}_\lambda.$  

In the previous section, we constructed a parabolic subgroup $Q=Q_\lambda$ of $G$ associated to $\{X_\lambda, Y_\lambda, H_\lambda\}.$ In our current setting, we need to somehow construct a parabolic subgroup in $K$. Following \cite[Sections 10-11]{AdamsVogan2021}, it turns out we can say something more about the parabolic constructed from our $\lie{sl}(2)$-triple: 

\begin{lemma}\label{Theta_stable_parabolic_real} 
	Let $X_\lambda\in \mathcal{N}_\theta.$ The associated parabolic subgroup $Q_\lambda$ is $\theta$-stable. In fact, $\theta Q_\lambda=Q_\lambda.$   
\end{lemma}  
\begin{proof} 
	We first show this on the Lie algebra level. Recall that $\lie{q}_\lambda=\bigoplus_{i\geq 0} \lie{g}_i$ where the $\lie{g}_i$ are the weight spaces of the semisimple operator $\ad(H_\lambda).$ Notice that as $H_\lambda\in \lie{k}$ we have that $\theta(H_\lambda)=H_\lambda.$ Now, for any $A\in \lie{g}_i$ we have that \[ [H_\lambda, \theta(A)]=[\theta(H_\lambda),\theta(A)]=\theta([H_\lambda,A])=i\theta(A)\]
	Hence, $\theta(A)\in \lie{g}_i$ and $\theta|_{\lie{g}_i}:\lie{g}_i\to \lie{g}_i$ is a vector space isomorphism. Whence, \[\theta\lie{q}_\lambda=\theta\left(\bigoplus_{i\geq 0} \lie{g}_i \right)=\bigoplus_{i\geq 0} \theta( \lie{g}_i )=\bigoplus_{i\geq 0} \lie{g}_i=\lie{q}_\lambda\]
	
	On the group level, consider $Q_\lambda\cap \theta Q_\lambda.$ This is a closed algebraic subgroup of $Q_\lambda$ and its Lie algebra is $\lie{q}_\lambda\cap \theta \lie{q}_\lambda=\lie{q}_\lambda.$ Therefore, $Q_\lambda\cap \theta Q_\lambda$ is a closed algebraic subgroup of $Q_\lambda$ with the same Lie algebra. Thus, it is also an open subgroup of $Q_\lambda.$ Hence, \[ Q_\lambda=\theta Q_\lambda\]
\end{proof} 

\begin{lemma} 
	Let $L_\lambda$ be the connected subgroup of $Q_\lambda$ with Lie algebra $\lie{l}_\lambda:=\lie{g}_0$ (the Levi factor). Then $L_\lambda$ is $\theta$-stable and $L_\lambda=\theta L_\lambda.$ 
\end{lemma} 
\begin{proof} 
	The Lie algebra $\lie{l}$ is $\theta$-stable by the proof given in the previous lemma. Similarly, $L_\lambda =\theta L_\lambda$ by the argument above simply by replacing $Q_\lambda$ in the proof. 
\end{proof} 

This gives a Levi decomposition $Q_\lambda=L_\lambda U_\lambda$ by $\theta$-stable subgroups. One remarkable property of $\theta$-stable parabolic subgroups is the following proposition: 

\begin{proposition}\label{Theta_Stable_Parabolics_Descend}\cite[Lemma 4.1]{Hecht1987}
	Let $P\subseteq G$ be a $\theta$-stable parabolic subgroup. Then $P\cap K$ is parabolic in $K.$ 
\end{proposition} 

We then have the following immediate corollary for $Q_\lambda.$ 
\begin{corollary} 
	$K\cap Q_\lambda$ is parabolic in $K.$ 
\end{corollary} 
\begin{proof}
	$Q_\lambda$ is $\theta$-stable by Lemma \ref{Theta_stable_parabolic_real}. Thus, we can apply the previous proposition and we are done.  
\end{proof} 

We may now repeat the processes from the previous sub-sections replacing $G$ by $K$ and $Q$ by $K\cap Q_\lambda.$ For ease of notation, set once and for all $Q_K:=K\cap Q_\lambda, L_K:=K\cap L_\lambda, U_K=K\cap U_\lambda.$ Recall the Cartan decomposition of the Lie algebra induced by the Cartan involution: 
\[ \lie{g}=\lie{k}\ds \lie{p} \]  
As $[\lie{k},\lie{p}]\subseteq \lie{p},$ we have that $\ad H_\lambda$ induces a grading (as above) on $\lie{p}$ which we shall write as \[ \lie{p}=\bigoplus_{i\in \Z} \lie{p}_i \]
with $\lie{p}_i=\{ S\in \lie{p}: \ad_{H_\lambda} (S)=iS\}.$ With the same conventions as above, set $\lie{p}^2=\bigoplus_{i\geq 2} \lie{p}_i.$ Notice that $X_\lambda\in \lie{p}_2\subseteq \lie{p}^2.$   

In fact, we can say something stronger about this grading: 
\begin{lemma}\label{Graded_Cartan_decomposition} 
	The Cartan decomposition descends to a graded involution of $\lie{g}$ under $\ad_{H_\lambda}.$  
\end{lemma} 		
\begin{proof} 
	As we observed above, $\theta(\lie{g}_i)=\lie{g}_i$ and $\theta^2=1.$ Therefore, for each $i,$ set $\lie{c}_i=\lie{g}_i^\theta$ and $\lie{d}_i=\lie{g}_i^{-\theta}$ the $+1$ and $-1$ eigenspaces respectively. We shall show that in fact, $\lie{c}_i=\lie{g}_i\cap \lie{k}$ and $\lie{d}_i=\lie{g}_i\cap \lie{p}=\lie{p}_i.$ Notice that as $H_\lambda\in \lie{k},$ we have that \begin{align*} \ad_{H_\lambda}(\lie{k})&\subseteq \lie{k}\\
	 \ad_{H_\lambda}(\lie{p})&\subseteq \lie{p}
	\end{align*} 
	Thus, $\lie{k}$ and $\lie{p}$ are $\ad_{H_\lambda}$-invariant subspaces of $\lie{g}$ and thus inherit gradings $\lie{k}_i$ and $\lie{p}_i$ via intersection. It is then immediately clear from the definitions that $\lie{c}_i=\lie{k}_i$ and $\lie{d}_i=\lie{p}_i.$ Hence, we have a graded Cartan decomposition \[\lie{g}=\lie{k}\ds \lie{p}=\bigoplus_{i\in \Z} \lie{k}_i\ds \lie{p}_i.\]
	This satisfies the following relations:  \begin{align*} 
		[\lie{k}_i,\lie{p}_j]&\subseteq \lie{p}_{i+j} \\
		[\lie{k}_i,\lie{k}_j]&\subseteq \lie{k}_{i+j}\\
		[\lie{p}_i,\lie{p}_j]&\subseteq \lie{k}_{i+j}
	\end{align*} 
\end{proof} 

We now have a similar result to Lemma \ref{Dense_Q_orbit_g_2}: 

\begin{lemma} 
	$Q_K\cdot X_\lambda\subseteq \lie{p}^2.$ Further, the orbit is open and dense. 
\end{lemma} 
\begin{proof} 
	As in the proof of Lemma \ref{Dense_Q_orbit_g_2}, it suffices to show that $U_K\cdot X_\lambda $ and $L_K\cdot X_\lambda$ are contained in $\lie{p}^2$ and that one of their orbits is dense. For $U_K$ we know again that $U_K$ is unipotent and $U_K\cong \lie{u}_K$ via the exponential map. Therefore, \[ \Ad(u)X_\lambda\in [\lie{u}_K,\lie{p}^2]\subseteq \lie{p}^3 \] where $\lie{u}_K=\lie{u}\cap \lie{k}=\bigoplus_{i>0} \lie{k}_i.$ Additionally, as $[X_\lambda,\lie{u}]=\lie{g}^3$ and $\lie{u}=\lie{k}^1\ds \lie{p}^1$ (by Lemma \ref{Graded_Cartan_decomposition}) we see that we must have $[X_\lambda,\lie{p}^1]\subseteq \lie{k}^3.$ It then follows that \[ [X_\lambda, \lie{u}_K]=\lie{p}^3 \]
	Whence, we see that $U\cdot X_\lambda=X_\lambda+\lie{p}^3.$ 
	
	Now, as before we compute that \[ [H_\lambda, \Ad(l)X_\lambda]=[\Ad(l)H_\lambda,\Ad(l)X_\lambda]=\Ad(l)[H_\lambda, X_\lambda]=2\Ad(l)X_\lambda\]
	for all $l\in L_K.$ Thus, $L_K\cdot X_\lambda \subseteq \lie{p}_2$ and we have proved that $Q_K\cdot X_\lambda\subseteq \lie{p}^2.$ 
	
	To show this orbit is dense, we shall show it has the same dimension as $\lie{p}_2.$ Consider the map \[ T:\lie{p}_2\to \Hom(\lie{p}_{-2}, \lie{k}_0) \]\[T(A)=\ad_A \]   
	Then $\ker T(X_\lambda)=\lie{p}_{-2}\cap \lie{g}^{X_\lambda}=0$ as the Lie algebra of the stabilizer is contained in $\lie{q}.$  This means we have an injection (hence an isomorphism) \[ [X_\lambda,\lie{p}_{-2}]\cong \lie{p}_{-2} \]
	Thus, $\dim [X_\lambda,\lie{p}_{-2}]=\dim \lie{p}_{-2}=\dim \lie{p}_2$ (by the representation theory of $\lie{sl}(2,\C)$). 
	
	Now, recall the invariant bilinear form $\psi$ from above.  We have \[ 0=\psi([\lie{k}^{X_\lambda}\cap \lie{k}_0,X_\lambda],\lie{p}_{-2})=\psi(\lie{k}^{X_\lambda}\cap \lie{k}_0, [X_\lambda, \lie{p}_{-2}])\]
	By nondegeneracy, we see that $[X_\lambda,\lie{p}_{-2}]\subseteq (\lie{k}^{X_\lambda}\cap \lie{k}_0)^\perp\cap \lie{k}_0.$  It follows at once that \[ \dim \lie{p}_{2}\leq \dim \lie{k}_0-\dim (\lie{k}^{X_\lambda}\cap \lie{k}_0)\]
	Whence, \[ \dim L_K\cdot X_\lambda= \dim \lie{k}_0-\dim (\lie{k}^{X_\lambda}\cap \lie{k}_0)\geq \dim\lie{p}_2\] 
	and $\dim L_K\cdot X_\lambda=\dim \lie{p}_{2}.$ Thus, $\overline{L_K\cdot X_\lambda}=\lie{p}_2.$ Orbits of affine algebraic groups are open in their closure and thus $L_K\cdot X_\lambda$ is open and dense. 
	
	 Combining this all, we. see that $Q_K\cdot X_\lambda$ is open and dense in $\lie{p}_2$ and is precisely given by \[ Q_K\cdot X_\lambda=L_K\cdot X_\lambda+\lie{p}^3\]	     
\end{proof}
		
Continuing in the same fashion, we want to construct a resolution of singularities of the closure of $K\cdot X.$ 
Set $\mathcal{R}_K:=K\times_{Q_K} \lie{p}^2.$ Then we have a map \[ \mu_K:\mathcal{R}_K\to \mathcal{N}_\theta \]
\[ \mu_K(k,\xi)\mapsto \Ad(k)\xi \]

\begin{proposition} 
		The map $\mu_K$ is a proper, birational map and its image is $\overline{K\cdot X_\lambda}.$ 
\end{proposition} 
\begin{proof} 
	Define a morphism $m_K:\mathcal{R}_K\to K/Q_K\times \mathcal{N}_\theta$ by $(k,X)\mapsto (kQ_K, \Ad(k)(X))$. We claim this map is injective. To see this, suppose $(hQ_K,\Ad(h)(Y))=(gQ_K,\Ad(g)(X))$ for some $(h,Y), (g,X)\in \mathcal{R}_K.$ Then by definition $hQ_K=gQ_K$ if and only if there exists $q\in Q_K$ such that $g=hq$ or equivalently that $h^{-1}g=q\in Q_K.$ Additionally, $\Ad(h)^{-1}\Ad(g)X=Y.$ Now, \[ (g,X)\sim (hq,X)\sim (h,\Ad(q)X)\sim(h,\Ad(h)^{-1}\Ad(g)X)=(h,Y) \]
	Hence, $m_K$ is injective. Further, its image is closed as it is $K/Q_K\times \overline{\mathcal{O}^\lie{k}_\lambda}.$ Thus $m_K$ is proper.  We then see that the moment map $\mu_K$ is simply the restriction of the projection map $\pi_2:K/Q_K\times \mathcal{N}_\theta\to \mathcal{N}_\theta.$ As $K/Q_K$ is projective (hence proper over $\C$), $\pi_2$ is projective. By \cite[Theorem 4.9]{Hartshorne1977}, projective morphisms of noetherian schemes are proper. Hence, $\mu_K$ is proper. 
	
	To show it is birational, notice that $K\cdot X_\lambda$ is open in the image. Now, the fibre of $\mu_K$ over $X_\lambda$ is $K^{X_\lambda}/(K^{X_\lambda}\cap Q_K)=\{*\}$. Whence, on $K\cdot X_\lambda$ we can define a rational inverse, namely $X \mapsto (k,X_\lambda)$ for any $k$ such that $\Ad(k)X_\lambda=X.$ This is well defined up to right multiplication by the stabilizer $K^{X_{\lambda}}$ which is contained in $Q_K.$   Finally, as $\mu_K$ is proper, the cornormal bundle is irreducible, and $K\cdot X_\lambda$ is dense in the image, we have that the image is $\overline{K\cdot X_\lambda}.$ This completes the proof.   
\end{proof} 

\begin{corollary}\label{Isomorphism_Equivariant_Resolution_K}  
	We have an isomorphism of equivariant $K$-groups \[ \mathbb{K}^K(\mathcal{R}_K)\cong \mathbb{K}^{Q_K}(\lie{p}^2)\cong \mathbb{K}(Q_K) \] 
\end{corollary} 
\begin{proof} 
	We again apply Thomason's result.  We have that \[\mathbb{K}^K(\mathcal{R}_K)\cong \mathbb{K}^{Q_K}(\lie{u})\]
	Now by \cite[Corollary 4.2]{Thomason1987} we have that if $G$ acts on affine $n$-space linearly, $\mathbb{K}^G(X)\cong \mathbb{K}^G(X\times \mathbb{A}_\C^n)$ for all $n$ and all $X.$ Setting $X=\{*\},$ we obtain that \[ \mathbb{K}^{Q_K}(\lie{u})\cong \mathbb{K}^{Q_K}(\{*\})=\mathbb{K}(Q_K).\]
	This completes the proof.  
\end{proof}

\section{Proof of the Main Theorem}

Let $G$ be a connected complex reductive Lie group. Let $K$ be the identity component of the set of fixed points of a Cartan involution $\theta$. Denote by $\lie{g}=\lie{k}\ds \lie{p}$ the associated Cartan decomposition on the Lie algebra level. Let $\lie{q}$ be a $\theta$-stable parabolic subalgebra of $\lie{g}$ and $Q$ the connected subgroup of $G$ with Lie algebra $\lie{q}.$ Denote by $S=K\cdot \lie{q}$ a closed $K$-orbit in the partial flag variety $X_Q=G/Q.$ It is a smooth subvariety. Put $T^*_S X_Q$ the conormal bundle to $S.$ It is a smooth, Lagrangian subvariety of the cotangent bundle $T^*X_Q.$ Fix a non-degenerate $G$-equivariant symmetric bilinear form $\psi$ on $\lie{g}.$ Then as a $K$-equivariant vector bundle, \[T^*_SX_Q\simeq K\times_{Q\cap K} (\lie{u}\cap \lie{p}) \]    
Denote by $\pi_K:T^*_S X_Q\to K/Q\cap K$ the canonical projection. Denote by $\mathcal{L}_{K/Q\cap K}(V)$ the equivariant vector bundle on $K/Q\cap K$ with fibre $V.$ 

\begin{theorem}\label{Canonical_Bundle}  
	$\omega_{T^*_SX_Q}\simeq \pi_K^*\mathcal{O}(2\rho(\lie{u}\cap \lie{p})-2\rho(\lie{u}\cap \lie{k})).$ Further, consider the line bundles \[Y=K\times_{Q\cap K} [(\lie{u}\cap \lie{p})\ds \C_{\lambda'}]\] with $\lambda'=\lambda+2\rho(\lie{u}\cap \lie{k})-2\rho(\lie{u}\cap \lie{p})\in \mathbb{W}(Q\cap K),$ then $\omega_Y\simeq \pi_K^*\mathcal{O}(\lambda).$  
\end{theorem} 	 

\begin{proof} 
	There is an exact sequence of $K$-equivariant vector bundles \[ 0\to \pi_K^*\mathcal{L}_{K/Q\cap K}(\lie{u}\cap \lie{p})^* \to \Omega_{T^*_SX_Q}\to \pi_K^* \mathcal{L}_{K/Q\cap K}(\lie{u}\cap \lie{k})\to 0\] which is fibre-wise split. Taking top exterior powers, we get that the fibre at the identity is \[ \C_{2\rho(\lie{u}\cap \lie{k})}\tensor \C^*_{2\rho(\lie{u}\cap \lie{p})}\simeq \C^*_{2\rho(\lie{u}\cap \lie{p})-2\rho(\lie{u}\cap \lie{k})}.\] Therefore, \[ \omega_{T^*_SX_Q}\simeq \pi_K^*\mathcal{O}(2\rho(\lie{u}\cap \lie{p})-2\rho(\lie{u}\cap \lie{k}))\]
	For the corresponding statement about line bundles, we additionally have a factor of $\C^*_{\lambda'}$ in the fibre. As $\lambda'=\lambda+2\rho(\lie{u}\cap \lie{k})-2\rho(\lie{u}\cap \lie{p}),$ we get that the fibre of $\omega_Y$ at the identity is $\C^*_\lambda.$ This completes the proof.    
\end{proof}

By the results of the introduction, we obtain a resolution of singularities  \[\begin{tikzcd} \KSpRes'=K\times_{Q_K} (\lie{u}\cap \lie{p}) \arrow[d,swap,"\mu_K"]\arrow[r,"\pi_K"]  & K/Q_K \\ \mathcal{N}_\theta'    \end{tikzcd}  \]

Let the Levi decomposition of $Q_K=L_KU_K$ be as above and denote by $\lie{q}_K, \lie{l}_K,$ and $\lie{u}_K$ the associated Lie algebras. For ease of notation, we will denote by $\mathcal{O}_{\KSpRes'}(\mu):=\pi_K^*\mathcal{O}_{K/Q_K}(\mu)=\pi_K^*\mathcal{L}_{K/Q_K}(\C^*_\mu)$ for any character of $L_K.$ The group of virtual characters of $L_K$ is generated by $W_{L_K}$ orbits of elements in $\mathbb{W}(Q_K).$  

\begin{example} 

Let $G_\R=PSL(n,\HH).$ Then the following versions of the above theorem hold. 

\begin{enumerate} 
	  \item Suppose $n=2k$ is even. Then $\omega_{\widetilde{\mathcal{N}_\theta}} \simeq \pi^*\mathcal{O}_{K/Q_K}(-2e_1-2e_2-...-2e_n)$ where $e_i$ are the coordinate functionals in $\lie{h}^*.$

	\item Suppose $n=2k+1$ is odd. Then $\omega_{\widetilde{\mathcal{N}_\theta}} \simeq \pi_K^*\mathcal{O}_{K/Q_K}(-2e_1-2e_2-...-2e_{n-1}).$

	\item Let $\lambda'=\lambda+2e_1+...+2e_n$ (respectively $\lambda+2e_1+...+2e_{n-1}$). Consider $Y=K\times_{Q_K} (\lie{p}^2\ds \C_{\lambda'}).$ Denote by $\pi_Y:Y\to K/Q_K.$ Then $\omega_U\simeq \pi_U^*\mathcal{O}_{K/Q_K}(\lambda).$

\end{enumerate} 

\end{example} 

Restricting Theorem \ref{Canonical_Bundle}, we obtain the following for the $K$-Springer Resolution:

\begin{corollary}\label{Nilpotent_Orbit_Case}
	$\omega_{\KSpRes'}\simeq \pi_K^*\mathcal{O}_{K/Q_K}(2\rho(\lie{u}\cap \lie{p})-2\rho(\lie{u}_K)).$ Further, if $Y=K\times_{Q_K} [(\lie{u}\cap \lie{p})\ds \C_{\lambda'}]$ with $\lambda'=\lambda+2\rho(\lie{u}_K)-2\rho(\lie{u}\cap \lie{p}),$ then $\omega_Y\simeq \pi_K^*\mathcal{O}(\lambda).$ 
\end{corollary}

We are now equipped to prove the vanishing result. Recall now the result of Grauert-Riemenschneider as given by Kempf's version \cite{Kempf1976}. 
\begin{proposition} 
	Let $Y$ be an algebraic variety and $\omega_Y$ the canonical bundle. If there exists a proper generically finite morphism $Y\to X$ where $X$ is an affine variety then $H^i(Y,\omega_Y)=0$ for $i>0.$ 
\end{proposition}  

Fix $\lambda\in \Lambda_K$ (the weight lattice) such that $\lambda'=\lambda+2\rho(\lie{u}_K)-2\rho(\lie{u}\cap \lie{p})$ is $Q\cap K$-dominant. 

\begin{lemma} 
	For $Y=K\times_{Q_K} [(\lie{u}\cap \lie{p})\ds \C_{\lambda'}]$ as above, $H^i(Y,\omega_Y)=0$ for $i>0.$ 
\end{lemma} 
\begin{proof} 
	Let $V_{\lambda'}$ be the irreducible representation of $K$ with highest weight $\lambda'.$ Consider the map $Y\to \mathcal{N}_\theta\times V_{\lambda'}$ given by $(k,(p,v))\mapsto (\mu(k,p), k\cdot v).$ We claim that this map satisfies the conditions of the above proposition. It is proper as $\mu$ is proper and the inclusion of a closed subspace is proper. To show it is generically finite, it suffices to exhibit a single point with finite pre-image. Consider $(p,0),$ with $K\cdot p$ the unique open orbit in $\mathcal{N}_\theta.$  As $\mu$ is a resolution of singularities over this orbit closure, it is an isomorphism over the open orbit. Hence the fibre is a singleton. Invoking the previous proposition, this proves the lemma.      
\end{proof} 

We now combine these lemmata to prove the following vanishing result: 
\begin{theorem}[Vanishing of higher cohomology]\label{Vanishing_of_Higher_Cohomology} 
	Let $\lambda'=\lambda+2\rho(\lie{u}_K)-2\rho(\lie{u}\cap \lie{p})\in \mathbb{W}(Q_K)$. Then \[H^i(\KSpRes',\mathcal{O}_{\KSpRes'}(\lambda'))=0\] for $i>0.$ 
\end{theorem} 
\begin{proof} 
	Let $Y=K\times_{Q_K} [(\lie{u}\cap \lie{p})\ds \C_{\lambda'}].$ From the above, we have that $0=H^i(Y,\omega_Y).$ Using the projection formula \cite[\href{https://stacks.math.columbia.edu/tag/01E6}{Tag 01E6}]{stacks-project}, we obtain an isomorphism \[ 0=H^i(Y,\omega_Y)\cong H^i(K/Q_K, \mathcal{O}(\lambda)\tensor \pi_*\mathcal{O}_{\lie{p}^2\ds \C_{\lambda'}})\cong H^i(K/Q_K,\mathcal{L}(\C^*_{\lambda}\tensor \operatorname{Sym}^\bullet(\lie{p}^2\ds \C_{\lambda'})^*))\]
	As $\operatorname{Sym}^k(V\ds W)\cong \bigoplus_{i+j=k}\Sym^i(V) \tensor \Sym^j (W),$ we have that \[ H^i(K/Q_K,\mathcal{L}(\C^*_{\lambda}\tensor \operatorname{Sym}^\bullet((\lie{p}^2)^*\ds \C_{\lambda'}^*)))=\bigoplus_{k,l\geq 0} H^i(K/B_K,\mathcal{L}(\C^*_{\lambda}\tensor \operatorname{Sym}^k(\lie{p}^2)^*\tensor \Sym^l(\C_{\lambda'}^*)))\] The weights of the symmetric algebra of $\C_{\lambda'}$ are $n\lambda'$ for $n\geq 0,$ we have \[  \bigoplus_{k,l\geq 0} H^i(K/Q_K,\mathcal{L}(\C^*_{\lambda'}\tensor \operatorname{Sym}^k((\lie{p}^2)^*)\tensor \Sym^l(\C_{\lambda'}^*)))=\bigoplus_{k,l\geq 0} H^i(K/Q_K,\mathcal{L}(\C^*_{\lambda'+l\lambda'}\tensor \operatorname{Sym}^k((\lie{p}^2)^*)))\]
	Now, using the projection formula again, we obtain \[ 0=\bigoplus_{k,l\geq 0} H^i(K/Q_K,\C^*_{\lambda'+l\lambda'}\tensor \operatorname{Sym}^k((\lie{p}^2)^*)=\bigoplus_{n\geq 0} H^i(\KSpRes', \pi^*\mathcal{L}_{K/B_K}(\C^*_{\lambda'+n\lambda'}))\]
	The $n=0$ summand is the desired cohomology group. This proves the result.  
\end{proof} 

From this proof, we immediately see the following general result:

\begin{corollary} 
	Let $\mathcal{O}$ be a $K$-orbit on $\mathcal{N}_\theta.$ Then we have that for $\lambda'=\lambda+2\rho(\lie{u}_K)-2\rho(\lie{p}^2)\in \mathbb{W}(Q_K)$  \[H^i(\widetilde{\mathcal{O}},\mathcal{O}_{\widetilde{\mathcal{O}}}(\lambda'))=0\;\;\; i>0.\]  
\end{corollary} 

In particular, by setting $\lambda=2\rho(\lie{u}\cap \lie{p})-2\rho(\lie{u}\cap \lie{k}),$ we get: 
\begin{corollary} 
	$R^i(\mu_K)_*\mathcal{O}_{\KSpRes'}=0$ for $i>0$ and thus the normalization of $\mathcal{N}_\theta'$ has rational singularities. Applying this reasoning to the previous corollary, we see that the normalization of any $K$-orbit $\mathcal{O}$ has rational singularities.   
\end{corollary}

\section{Normalization of the Nilpotent cone} 
One deficit of the above result is that at first glance we only obtain statements about each $K$-orbit and not the entire $K_\theta$-nilpotent cone. However, using the interplay between the geometry of $\mathcal{N}_\theta$ and the various $\mathcal{N}_\theta'$ we can indeed say something about the entire $K_\theta$-nilpotent cone.  

Recall that the (finitely many) irreducible components of $\mathcal{N}_\theta$ are precisely the closures of $K$-orbits of principal nilpotent elements \cite[Theorem 3]{KostantRallis1971}. Let $X_1,...,X_n$ be principal nilpotent elements such that $Y^i:=\close{K\cdot X_i}$ are the irreducible components of $\mathcal{N}_\theta.$ Similarly, denote by $(\KSpRes')_i$ the resolution of singularities of $Y_i.$ Recall the following geometric corollary of the Lasker-Noether theorem on primary decompositions:    

\begin{corollary}\cite[\href{https://stacks.math.columbia.edu/tag/0CDV}{Tag 0CDV}]{stacks-project} 
	Let $Y$ be a reducible algebraic variety and $Y=\displaystyle{\bigcup_{i=1}^n} Y^i$ a decomposition into irreducible components. Denote by $Y^\nu$ the normalization of $Y.$ Then $Y^\nu=\displaystyle{\bigsqcup_{i=1}^n} (Y^i)^\nu.$    
\end{corollary} 
\noindent By the above, we may write \[ \mathcal{N}_\theta^\nu=\bigsqcup_{i=1}^n (Y^i)^\nu\] If we let $\mathcal{O}_{i,\nu}$ denote the structure sheaf of $Y_i^\nu$, we have that $\mathcal{O}_{\mathcal{N}_\theta^\nu}=\bigoplus_{i=1}^n \mathcal{O}_{i,\nu}.$ In a similar fashion, we have a decomposition of $\KSpRes=\bigsqcup_{i=1}^n (\KSpRes')_i.$ Therefore, \[ \mathcal{O}_{\KSpRes}=\bigoplus_{i=1}^n \mathcal{O}_{(\KSpRes')_i}\]
As $\mu_K$ gives a resolution of singularities of $\mathcal{N}_\theta$ which restricts to a resolution of each $Y_i,$ we have that \[(\mu_K)_*\mathcal{O}_{(\KSpRes')_i}=\mathcal{O}_{i,\nu}\]   
Combining this with the vanishing theorem gives the following: 

\begin{corollary}\label{Rational_normal} 
	Let $G,K_\theta,$ and $K$ be as above. Then the following hold: 
	\begin{enumerate} 
		\item $H^i(\KSpRes, \mathcal{O}_{\KSpRes})=0$ for $i>0.$ 
		\item $(\mathcal{N}_\theta)^\nu$ has rational singularities. 
		\item If $\mathbb{O}$ be any $K_\theta$-orbit on $\mathcal{N}_\theta,$ then $\mathbb{O}^\nu$ has rational singularities.   
	\end{enumerate} 
\end{corollary} 

\begin{remark} We note that this generalizes a result of Hinich \cite{Hinich1991} for complex groups. In this line, we believe that the $K_\theta$-nilpotent cone is also Gorenstein however a proof eludes us at this time. 
\end{remark} 

\section{Further Generalization}
If we have some additional properties of $G_\R,$ it is possible to extend these results in the following way:

\begin{definition}\label{Quasi_complex}
	Let $G$ be a connected complex reductive group Lie group. We say that $G_\R$ is of \textbf{quasi-complex type} (\textbf{QCT}), if the following hold: 
		\begin{description} 
			\item[G-1)] The $K$-nilpotent cone $\mathcal{N}_\theta$ is the closure of a single $K$-orbit. 
			\item[G-2)] All $K$-orbits on $\mathcal{N}_\theta$ are even dimensional. 
		\end{description}  
		
	If $G_\R$ only satisfies $\textbf{G-1}$ we say $G_\R$ is of $\textbf{quasi-adjoint type}$ or $\textbf{QAT}$ for short. \\
\end{definition} 	
\noindent Notice that \textbf{G-2} is equivalent to 
	\begin{description} 
		\item[G-2')] For any $K$-orbit $\mathcal{O}_{K}$ the associated complex $G$-orbit $\mathcal{O}_G$ has dimension divisible by 4. 
	\end{description} 
	
	\begin{remark} As shown above, if $G=GL(2n,\C)$ and $\theta$ chosen so $K=Sp(2n,\C)$, then it is of QCT. In particular, all simple complex groups are of QCT. We believe a complete list of groups in this class are as follows: all simple complex groups, $GL(n,\HH)$, $Sp(p,q)$, $SO^*(2n)$, and rank 1 $E_6.$ Morally, the classical real groups appearing here arise as a result of a "quaternionic" symplectic structure on on the complex $G$-orbits. \end{remark} 
	
	Finally, we have the following: 
		
\begin{theorem} 
	Let $G$ be QAT. Then the following hold: 
		\begin{enumerate} 
			\item If $Q=LU$ denotes the $\theta$-stable parabolic associated to the principal $K$-orbit on $\mathcal{N}_\theta,$ then \[\omega_{\KSpRes}\simeq \pi^*\mathcal{O}_{K/Q_K}(2\rho(\lie{u}\cap \lie{p})-2\rho(\lie{u}\cap \lie{k}))   \]
			\item For $\lambda'=\lambda-2\rho(\lie{u}\cap \lie{p})+2\rho(\lie{u}\cap \lie{k}) \in \mathbb{X}_+^*(L)$ we have \[H^i(\KSpRes,\mathcal{O}_{\KSpRes}(\lambda'))=0 \text{ for } i>0 \] 
			\item If $G$ is in fact $QCT$ then $\mathcal{N}_\theta$ is a complete intersection normal variety and \[ \textbf{R}(\mu_K)_*\mathcal{O}_{\KSpRes}=\mathcal{O}_{\mathcal{N}_\theta} \]
			In particular, $\mathcal{N}_\theta$ has rational singularities. 
		\end{enumerate} 
\end{theorem} 
\begin{remark} 
	Notice that if a real form $G_0$ of $G$ has a single conjugacy class of Cartan subgroups, then the shift in the canonical bundle is given by negating the sum of the (necessarily compact) imaginary roots.  
\end{remark} 

Using results on cohomological induction, we obtain the following characterization of global functions on $\mathcal{N}_\theta$: 

\begin{corollary} 
	Let $G$ be QCT. Then \[ \Gamma(\mathcal{N}_\theta,\mathcal{O}_{\mathcal{N}_\theta})|_K\cong A_\lie{q}(-2\rho(\lie{u}\cap \lie{p}))|_K \] 
	
\end{corollary} 
\begin{proof}
	The right hand side is a characterization of functions on $\KSpRes.$ This follows from a multiplicity calculation and the $K$-types of $A_\lie{q}(\lambda)$ modules (see \cite[Section 2]{VoganZuckerman1984}). 
\end{proof}

\bibliographystyle{alpha} 
\bibliography{Representation_Theory_references_Master}

\end{document}